\documentclass[10pt]{amsart}
\usepackage{a4,amssymb,times,amsmath,cancel}
\usepackage{graphicx,color,epsfig}

\usepackage[all]{xy}
\usepackage{tikz}
\usetikzlibrary{calc}

\def\ot{\leftarrow}

\let\cal\mathcal

\def\cC{{\cal C}}
\def\cD{{\cal D}}

\def\cI{{\cal I}}
\def\cJ{{\cal J}}

\def\cP{{\cal P}}
\def\cQ{{\cal Q}}
\def\cR{{\cal R}}
\def\cS{{\cal S}}

\let\blb\mathbb

\def \ZZ{{\blb Z}}

\def \TT{{\blb T}}

\def \Rl{{\blb R}}

\def \VV{{\blb V}}
\def \SS{{\blb S}}

\newcommand{\se}[1]{\begin{equation*}\begin{split}#1\end{split}\end{equation*}}

\newcommand{\del}{\cancel}
\newcommand{\C}{\mathbb{C}}
\newcommand{\N}{\mathbb{N}}
\newcommand{\Z}{\mathbb{Z}}
\newcommand{\R}{\mathbb{R}}

\newcommand{\wt}{*+{\circ}}
\newcommand{\bk}{*+{\bullet}}
\newcommand{\grp}[1]{\mathsf{#1}}
\newcommand{\vtx}[1]{*+[o][F-]{\scriptscriptstyle #1}}

\newcommand{\Tr}{\textrm{Tr}}

\newcommand{\cycle}{\circlearrowright}

\newcommand{\srep}{\ensuremath{\mathsf{srep}}}
\newcommand{\<}{\langle}
\renewcommand{\>}{\rangle}

\newcommand{\lift}{\ensuremath{\mathsf{lift}}}

\newcommand{\Mod}{\ensuremath{\mathsf{Mod}}}
\newcommand{\Mor}{\ensuremath{\mathsf{Mor}}}

\newcommand{\coeff}{\text{cf}}

\newcommand{\Ext}{\mathsf{Ext}}
\newcommand{\Eqv}{\mathsf{Eqv}}
\newcommand{\Inv}{\mathsf{Inv}}
\newcommand{\Dim}{\mathsf{Dim \,}}
\newcommand{\image}{\mathsf{Im \,}}

\newtheorem{lemma}{Lemma}[section]

\newtheorem{theorem}[lemma]{Theorem}
\newtheorem{corollary}[lemma]{Corollary}
\newtheorem{property}[lemma]{Property}

\theoremstyle{definition}

\newtheorem{example}[lemma]{Example}
\newtheorem{definition}[lemma]{Definition}

{

}

\theoremstyle{remark}

\newtheorem{remark}[lemma]{Remark}

\newcommand{\trep}{\ensuremath{\mathsf{trep}}}
\newcommand{\Trep}{\ensuremath{\mathsf{trep}}}

\newcommand{\Mat}{\mathsf{Mat}}

\newcommand{\Hom}{\textrm{Hom}}

\newcommand{\Ker}{\textrm{Ker}}

\newcommand{\GL}{\ensuremath{\mathsf{GL}}}

\newcommand{\qpol}{\cQ}
\newcommand{\carr}{\mathtt R}
\newcommand{\ccyc}{\mathtt E}

\title{Calabi-Yau algebras and weighted quiver polyhedra}

\author{Raf Bocklandt}
\address{Raf Bocklandt\\
School of Mathematics and Statistics\\
Herschel Building\\
Newcastle University\\
Newcastle upon Tyne\\
NE1 7RU\\
UK}
\email{raf.bocklandt@gmail.com}

\xyoption{all}

\begin{document}
\begin{abstract}
Dimer models have been used in string theory to construct path algebras with relations
that are $3$-dimensional Calabi-Yau Algebras. These constructions result in algebras that share some specific 
properties: they are finitely generated modules over their centers and their representation spaces are toric
varieties. In order to describe these algebras we introduce the notion of a toric order and
show that all toric orders which are $3$-dimensional Calabi-Yau algebras can be constructed from dimer models on a torus.

Toric orders are examples of a much broader class of algebras: positively graded cancellation algebras. 
For these algebras the CY-3 condition implies the existence of a weighted quiver polyhedron, which is an extension of dimer models obtained by replacing the torus with any two-dimensional compact orientable orbifold. 
\end{abstract}
\maketitle

\section{Introduction}\label{intro}
Calabi-Yau algebras play an important role in theoretical physics because their derived categories can be used to
describe brane configurations in the $B$-model of topological string theory. There are several ways to construct examples of this kind of algebras such as McKay correspondence \cite{GB, CT} or exceptional 
sequences \cite{aspin}. Another important construction method are dimer models \cite{HanKen,FranHan,HanHer}. A dimer model $\cD$ consists of a bipartite graph (with black and white vertices) that is embedded in a compact surface. 
The corresponding algebra $A_\cD$ is the path algebra with relations of the dual graph oriented such that a cycle around
a black (white) vertex has a (anti-)clockwise orientation. The relations come from the partial derivatives of a superpotential which is the sum of all clockwise cycles minus the sum of all anti-clockwise cycles.

It was shown by Nathan Broomhead in \cite{broomhead}, by Sergey Mozgovoy and Markus Reineke in \cite{MR} and by Ben Davison in \cite{davison} that if the dimer model satisfies certain consistency conditions, the algebra $A_\cD$ is a $3$-dimensional Calabi-Yau Algebra. 

In this paper we will show why dimer models appear in this setting and to which extent they arise from the Calabi-Yau property.

The Calabi-Yau algebras that one obtains from dimer models on a torus share quite specific properties.
They are meant to be noncommutative toric resolutions of a toric variety and therefore
these algebras are prime and finitely generated modules over their centers, which are the coordinate rings of the
affine toric varieties one wishes to resolve. The fact that the resolution is supposed to be toric implies that
the algebra is a positively graded subalgebra of $\Mat_n(T)$ where $T=\C[\Z^k]$ is the coordinate ring of the torus inside the toric variety. 
We will call any such algebra a toric order and discuss how they fit in the notion of a noncommutative crepant resolution as introduced by Van den Bergh.

In this paper we will prove that if a positively graded toric order is CY-3 then it comes from a dimer model on a torus.
The way we prove this result is by generalizing both sides and proving a similar theorem in this generalized context.
On the one hand, we relax the definition of a toric order to cancellation algebras (Definition \ref{cancellation}) and on the other hand we introduce the notion of a weighted quiver polyhedron which corresponds roughly to a dimer model on a two-dimensional orientable orbifold (Examples can be found in section \ref{examples}). 
Our main theorem then states
\begin{theorem}[=Theorem \ref{qpol}]
Every positively graded cancellation algebra that is CY-3 comes from a weighted quiver polyhedron.
\end{theorem}
In the specific situation of toric orders, theorem \ref{qpolto} then shows that this quiver polyhedron
must in fact come from a dimer model.

The paper is structured as follows. After the preliminary sections on path algebras, Calabi-Yau algebras and noncommutative resolutions, we
introduce toric orders in section \ref{sectionto} and discuss how they fit in the theory of noncommutative resolutions.
In section \ref{sectionca} we generalize toric orders to cancellation algebras and discuss bimodule resolutions in this setup.
In section \ref{polyhedron} we define the combinatorial notion of a quiver polyhedron, relate it to dimer models and work out a theory
of Galois covers for them. In section \ref{maintheorem} we prove the main theorem. Section \ref{cancellation} contains a short discussion on the cancellation
property for quiver polyhedra and uses this to prove that toric CY-3 orders come from dimer models.
We end with some examples of quiver polyhedra and their corresponding Jacobi Algebras.

\section{Acknowledgements}
I wish to thank Jan Stienstra for the talk he gave in Antwerp \cite{stienstra}, 
which initiated my interest into the theory of dimer models. I also want to thank Nathan Broomhead, Alastair King, Michael Wemyss, Travis Schedler and Sergei Mozgovoy for the interesting discussions about these subjects at the Liegrits workshop in Oxford in January 2008 and during my stays in Bath, Chicago and Wuppertal.

\section{Preliminaries}\label{prelim}

\subsection{Path algebras with relations}
As usual a \emph{quiver} $Q$ is an oriented graph. We denote the set of vertices by $Q_0$, the set of arrows by $Q_1$ and the maps $h,t$ assign to each arrow its head and tail.
A \emph{nontrivial path} $p$ is a sequence of arrows $a_1\cdots a_k$ such that $t(a_i)=h(a_{i+1})$, whereas a \emph{trivial path} is just a vertex. We will denote the length of a path by $|p|:= k$ and the head and tail by $h(p)=h(a_1),~ t(p)=t(a_k)$. A path is called cyclic if $h(p)=t(p)$. 
Later on we will denote by $p[i]$ the $n-i^{th}$ arrow of $p$ and by $p[i\dots j]$ the subpath $p[i]\dots p[j]$.
\[
 \xymatrix{\vtx{}&\vtx{}\ar[l]_{p[n-1]}&\vtx{}\ar[l]_{p[n-2]}&\vtx{}\ar@{.>}[l]_{p[1]}&\vtx{}\ar[l]_{p[0]}}\text{ and }p = p[n-1]p[n-2]\dots p[1]p[0].
\]
A quiver is called \emph{connected} if it is not the disjoint union of two subquivers and it is \emph{strongly connected} if there is a cyclic path through each pair of vertices.

The \emph{path algebra} $\C Q$ is the complex vector space with as basis the paths in $Q$ and the multiplication of two paths $p$, $q$ is their concatenation $pq$ if $t(p)=h(q)$ or else $0$.
The span of all paths of nonzero length form an ideal which we denote by $\cJ$.
A \emph{path algebra with relations} $A=\C Q/\cI$ is the quotient of a path algebra by a finitely generated ideal $\cI \subset \cJ^2$. A path algebra is connected or strongly connected if and only if its underlying quiver is.

We will call a path algebra with relations $\C Q/\cI$ \emph{positively graded} if there exists a grading $\carr: Q_1 \to \Rl_{>0}$ such that $\cI$ is generated by homogeneous relations. Borrowing terminology from physics, we will sometimes call this map the $R$-charge.

A special type of path algebras with relations are Jacobi algebras. To define these we need to introduce some notation.
The vector space $\C Q/[\C Q,\C Q]$ has as basis the set of cyclic paths up to cyclic permutation of the arrows. We can embed this space into $\C Q$ by mapping a cyclic path onto the sum of all its possible cyclic permutations:
\[
\cycle : \C Q/[\C Q,\C Q] \to \C Q: a_1\cdots a_n \mapsto \sum_i a_i\cdots a_na_1\cdots a_{i-1}.
\] 
An element of the form $p +[\C Q,\C Q]$ where $p$ is a cyclic path will be called a \emph{cycle}. Usually we will drop the $+[\C Q,\C Q]$ from the notation and represent the cycle by one of its cyclic paths.

Another convention we will use is the deletion of arrows: if $p:= a_1\cdots a_n$ is a path and $b$ an arrow, then
$p \del b=a_1\cdots a_{n-1}$ if $b=a_n$ and zero otherwise. Similarly one can define $\del bp$. These new defined maps can be combined to obtain a 'derivation' 
\[
\partial_a : \C Q/[\C Q,\C Q] \to \C Q : p \mapsto \cycle(p)\del a =\del a \cycle(p).
\]
An element  $W \in \cJ^3/[\C Q,\C Q] \subset \C Q/[\C Q,\C Q]$ is called a \emph{superpotential}. This element does not need to be homogeneous. If we quotient out the partial derivatives of a superpotential we get an algebra which is called the \emph{Jacobi algebra}:
\[
 A_W := \C Q /\<\partial_a W: a \in Q_1\>.
\]
Note that if $W$ is homogeneous for some $R$-charge $\carr$, then the corresponding Jacobi Algebra is a positively graded algebra. The converse does not need to be true.

\subsection{Calabi-Yau Algebras}

\begin{definition}
A path algebra with relations $A$ is \emph{$n$-dimensional Calabi-Yau} (CY$-n$) 
if $A$ has a projective bimodule resolution $\cP^\bullet$ that is dual to its $n^{th}$ shift
\[
 \Hom_{A-A}(\cP^\bullet,A\otimes A)[n] \cong \cP^\bullet
\]
\end{definition}
For further details about this property we refer to \cite{bocklandt} and \cite{GB}.
In this paper we will only need the following results:

\begin{property}\label{CYprop}
If $A$ is CY-n then 
\begin{itemize}
\item[C1]
The global dimension of $A$ is $n$
\item[C2] 
If $X, Y \in \Mod A$ then
\[
\Ext^k_A(X,Y) \cong \Ext^{n-k}_A(Y,X)^*.
\]
\item[C3]
The identifications above give us a pairings $\<,\>^k_{XY}: \Ext^k_A(X,Y) \times \Ext^{n-k}_A(Y,X) \to \C$ which satisfy
\[
\<f,g\>^k_{XY} = \<1_X,g*f\>^0_{XX} = (-1)^{k(n-k)}\<1_Y,f*g\>^0_{YY}, 
\]
where $*$ denotes the standard composition of extensions.
\end{itemize}
\end{property}
Proofs can be found in \cite{bocklandt}.

\subsection{Noncommutative resolutions and orders}

Suppose $\VV$ is a normal variety with coordinate ring $R$ and function field $K$. A resolution of $\VV$ is a proper birational surjective map $\pi: \tilde \VV \to \VV$ such that $\tilde \VV$ is smooth.
The birationality of $\pi$ implies that it gives an isomorphism on the level of the function fields: $K(\tilde \VV)=K$.

A nice method to try to construct a resolution is by using orders. An \emph{$R$-order} in $\Mat_n(K)$ is an $R$-algebra $A\subset \Mat_n(K)$ that is a finitely generated $R$-module and
\[
A \cdot K =A \otimes_R K= \Mat_n(K).
\]
The embedding $R \subset A$ can be seen as a noncommutative generalization of the resolution because birationally (i.e. tensoring with $K$) it gives a Morita equivalence instead of an isomorphism.

Given an order $A$, we have a notion of a trace $\Tr: A \to R$, which is the restriction of the standard trace function in $\Mat_n(K)$. Traces of elements in $A$ sit in $R$ because $R$ is a normal domain.
This trace allows us to consider the $n$-dimensional trace preserving representations of $A$:
\[
 \Trep A := \{ \rho: A \to \Mat_n(\C)| \Tr \rho(a)=\rho(\Tr a)\}
\]
This object can be given the structure of an affine scheme (take care, it can consist of several components). It has an action of $\GL_n(\C)$ by conjugation and using this action we can reconstruct $A$ as the ring of equivariant maps and $R$ as the ring of invariant maps (see \cite{procesi}):
\se{
A &= \Eqv_{\GL_n} ( \Trep A, \Mat_n(\C)) := \{f: \Trep A \to \Mat_n(\C)|\forall g \in \GL_n: f(\rho^g) = f(\rho)^g\},\\
R &= \Inv_{\GL_n} ( \Trep A, \C) := \{f: \Trep A \to \C|\forall g \in \GL_n: f(\rho^g) = f(\rho)\}.
}
Geometrically this means that $R$ is the coordinate ring of the categorical quotient $\Trep A/\!\!/\GL_n$ and
this quotient parameterizes the \emph{isomorphism classes of semisimple trace preserving representations of $A$}.
In general the space $\Trep A$ consists of more than one component but there is only one component that maps surjectively 
onto the quotient. This is the component that contains the generic simples and we denote it by $\srep A$.

To construct a resolution of $\VV=\srep A/\!\!/\GL_n$, we can try to take a Mumford quotient instead of the categorical quotient. To do this, one must specify a stability condition, which in the case of path algebras with relations amounts to choosing a $\theta \in \Z^{Q_0}$ (see \cite{Kingstab}). The new quotient $\VV_\theta = \srep A/\!\!/_\theta\GL_n$ parameterizes the \emph{isomorphism classes of (direct sums of) $\theta$-stable trace preserving representations of $A$}. 
If one is lucky the new quotient is smooth and then it provides a resolution of $\VV$.

The idea of using orders for the construction of resolutions motivates the notion of a noncommutative resolution.
There are many possible definitions but they all share the following properties:
\begin{itemize}
\item $A$ is an $R$-order in $\Mat_n(K)$
\item $A$ has some smoothness property: finite homological dimension/homologically homogeneous/Calabi-Yau.
\end{itemize}
In this paper the focus is on the Calabi-Yau property, so for us noncommutative resolutions are Calabi-Yau orders.
Although at first sight this seems to be slightly different from the notion of the noncommutative crepant resolutions introduced by Van den Bergh
in \cite{nccrep}, in the 3-dimensional toric setting these notions will coincide (see \cite{bocklandtconsistency}).

\section{Toric orders}\label{sectionto}

If $\VV$ is a toric variety, then it has a faithful action of a torus $\TT^k = \C^{*k}$ with a dense open orbit.
Ring-theoretically this means that $R$ is $\Z^k$-graded and we can embed it as a graded subring of $T :=\C[\TT^k] = \C[X_1,X_1^{-1},\dots, X_k,X_k^{-1}]$.

To resolve the singularities of $\VV$, we want to keep the toric structure of $\VV$ so we need to construct a toric resolution. By this we mean that the map $\pi: \tilde \VV \to \VV$ is a $\C^{*k}$-equivariant map that is one to one on the torus $\TT^k$. From the point of view of rings, the coordinate ring of the torus now substitutes for the function field $K$ and everything gets $\Z^k$-graded.

This enables us to define \emph{toric orders}.
\begin{definition}
Let $R\subset T=\C[X_1,X_1^{-1},\dots, X_k,X_k^{-1}]$ be the coordinate ring of a toric variety. A toric $R$-order $A$ is a positively $\Z^k$-graded $R$-subalgebra of $\Mat_n(T)$ that is a finitely generated $R$-module and
\begin{itemize}
\item[TO1] $A \cdot T = \Mat_n(T)$
\item[TO2] $R^{\oplus n}\subset A$
\end{itemize}
\end{definition}
\begin{remark}
By positively $\Z^k$-graded, we mean there is a vector $u \in \Z^k$ such that if there are nonzero homogeneous elements with degree $v\in\Z^k \setminus\{0\}$, then $u\cdot v>0$.
\end{remark}
Toric orders are special orders, so we can also reconstruct $R$ and $A$ from the invariant and equivariant maps on $\Trep A$. If we do this we leave the toric context because $\GL_n$ is not toric. However, with a slight modification we
can make everything we said in the previous section work in the toric context.

We can get rid of $\GL_n$ by looking at $\alpha$-dimensional representations with $\alpha=(1,\dots,1)$. 
Because of condition TO2, the standard idempotents $e_i \subset \Mat_n(\C) \subset \Mat_n(T)$ must sit in $A$. We now define
\[
\trep_{\alpha} A := \{\rho \in \trep A| \rho(e_i)=e_i\}
\]
This is a closed subscheme of $\trep A$ that meets every orbit. The action of $\GL_n$ on $\trep A$ restricts to 
an action of ${\GL_{\alpha}} = \C^{*n} \subset \GL_n$ on $\trep_{\alpha} A$ and 
\[
 \trep A = \trep_{\alpha} A \times_{\GL_{\alpha}} \GL_n.
\]
Furthermore we have again that
\[
A = \Eqv_{\GL_{\alpha}} ( \Trep_{\alpha} A, \Mat_n(\C)) \text{ and }R = \Inv_{\GL_{\alpha}} ( \Trep_{\alpha} A, \C).
\]

Just as before we single out one component $\srep_{\alpha} A = \srep A \cap \trep_{\alpha} A$.
This component contains a $n-1+k$-dimensional torus coming from the pullback of the representations
of $\Mat_n(T)$ and there is a combined action of $\C^{*n-1}$ from $\GL_{\alpha}/\C^*$ and $\C^{*k}$ by scaling of the variables. Therefore $\srep_{\alpha} A$ can be seen as a toric variety but it is not necessarily normal. 

Unlike in the general case of orders, toric orders have the advantage that one only needs $\srep_{\alpha} A$ to
reconstruct the order and not the whole space $\trep_\alpha A$.

\begin{theorem}\label{srep}
If $A$ is a toric $R$-order in $\Mat_n(K)$ then
\[
A = \Eqv_{\GL_{\alpha}} ( \srep_{\alpha} A, \Mat_n(\C)) \text{ and }R = \Inv_{\GL_{\alpha}} ( \srep_{\alpha} A, \C).
\]
\end{theorem}
\begin{proof}
We have a map $\Eqv_{\GL_{\alpha}}( \trep_\alpha, \Mat_n(\C))  \to \Eqv_{\GL_{\alpha}}( \srep_\alpha A, \Mat_n(\C))$
by restriction. This map decomposes as a direct sum of maps according to the matrix entries
\[
\Eqv_{\GL_{\alpha}}( \trep_\alpha, \Mat_n(\C))_{ij}  \to \Eqv_{\GL_{\alpha}}( \srep_\alpha, \Mat_n(\C))_{ij}
\]
But $\Eqv_{\GL_{\alpha}}( \trep_\alpha A, \Mat_n(\C))_{ij}$ is the subspace $\C[\trep_{\alpha} A]_{ij} \subset \C[\trep_{\alpha} A]$ of weight $i-j$ for the $\GL_{\alpha}$-action. The same holds for $\Eqv_{\GL_{\alpha}}(\srep A, \Mat_n(\C))_{ij}$.

The map $\C[\trep_{\alpha} A] \to \C[\srep_{\alpha} A]$ is a surjection that is compatible with the $\C^{*n}$-action.
This means that $\C[\trep_{\alpha} A]_{ij} \to \C[\srep_{\alpha} A]_{ij}$ is surjective. 
$\C[\trep_{\alpha} A] \to \C[\srep_{\alpha} A]$ is not an injection but it becomes an injection if we tensor it over $R$
with the torus ring $T$ (note that $R$ sits both in $\C[\trep_{\alpha} A]$ and $\C[\srep_{\alpha} A]$ as a subring because $\srep_\alpha A$ is the component that maps surjectively to $\VV$). Therefore if $a \in \C[\trep_{\alpha} A]_{ij}$ sits in the kernel then we can lift $a$ to an element in $A$ such that $a \otimes_R 1_T =0$ but this is impossible because $A \subset \Mat_n(T)$.

The second statement follows directly from the first.
\end{proof}
\begin{remark}
This property corresponds to the notion of algebraic consistency, introduced by Broomhead. 
To be more specific $\Eqv_{\GL_{\alpha}} (\srep_{\alpha} A,\Mat_{n}(\C))$ can be identified with the algebra $B$ in \cite{broomhead}.
\end{remark}

\section{Category algebras and cancellation algebras}\label{sectionca}

In this section we will extend the notion of toric orders to a non-Noetherian setting. This generalization gives rise to the notion of a cancellation algebra.

\subsection{Motivation and definition}

From any toric algebra $A \subset \Mat_n(T)$ we can construct a category $\cC_A$. The objects
of this category are the elementary idempotents $e_i \subset A$ and the morphisms from $e_i$ to $e_j$
are the monomials of $T$ which occur on the $(i,j)$-th entries of elements in $A$. In other words:
\[
 \Hom_{\cC_A}(e_j,e_i) = \{\text{monomials in }e_iAe_j\}.
\]
We can reconstruct $A$ as the \emph{category algebra} of $\cC_A$. 
As a vector space, the basis of this algebra is
given by the set of all morphisms of $\cC_A$.
The multiplication is defined by the composition of morphisms if possible and zero otherwise. 

The category $\cC_A$ has a special property: it is a cancellation category.
\begin{definition}
A category $\cC$ is called a cancellation category if every morphism is epic and monic:
\[
 \forall a,b,c: ab=ac \implies b=c \text{ and }ac=bc \implies a=b \text{ (when defined).}
\]
The category algebra of a cancellation category is called a cancellation algebra.
\end{definition}
\begin{remark}
We will assume implicitly that a cancellation category only has a finite number of objects, 
and can be generated by a finite number of arrows. 
This is to make sure that the cancellation algebras we will consider are finitely generated as algebras. 
Later we will also investigate cancellation categories with a countable number of objects. 
In this case the category algebra will not be unital.
\end{remark}
\begin{remark}
A category algebra can equivalently be defined as a path algebra of a quiver with relations $\C Q/{\cI}$ where $\cI$ is generated by elements of the form $p-q$ with $p,q$ paths. Therefore it makes sense to use the notation $h(p)$ and $t(p)$ for morphisms in the category. In general it is not easy to check from a set of relations of the required form whether the corresponding category algebra is a cancellation algebra or not.

Just as for quivers we can speak of positively graded categories, category algebras and in particular positively graded toric orders. In this last case all monomials in a toric order are homogeneous for the grading $\carr$.
\end{remark}

The simplest examples of cancellation categories are groupoids. In these categories the cancellation law holds trivially because every morphism is invertible. Subcategories of groupoids are also cancellation categories, however
unlike in the case of abelian groups and semigroups it is not true that every cancellation category embeds in a groupoid.

If we return to the section of toric orders, it is easy to check that $\Mat_n(T)$ is the category algebra of the groupoid with 
$n$ objects $\{e_i\}_{i=1}^n$ such that $\Hom(e_i, e_j) = \ZZ^k$ for any $i$ and $j$. 
As a consequence, if $A$ is a toric order, then $\cC_A$ is a cancellation category because it is a subcategory of that groupoid.

\begin{lemma}
Every toric order is a cancellation algebra.
\end{lemma}

\subsection{Bimodule resolutions}\label{GPR}

Let $A$ be a category algebra with corresponding category $\cC$.
A (bi)-module $M$ of $A$ is called $\cC$-graded if $M = \bigoplus_{p \in \Mor\cC} M_p$ such that $M_p \subset h(p)Mt(p)$ and $\forall q: qM_p \subset M_{qp}$ and $M_pq \subset M_{pq}$, where we used the convention that $M_?=0$ to cover the case where $pq$ or $qp$ is not defined. A homogeneous map is a morphism $\phi :M \to N$ such that $\phi M_p \subset N_p$ and the kernel and image of a homogeneous map are clearly $\cC$-graded. For every $p$ in $\Mor \cC$ we define the projective bimodule $F_p= Ah(p) \otimes p \otimes t(p)A$ with the obvious grading $q_1\otimes p \otimes q_2 \in (F_p)_{q_1pq_2}$ and analogously the projective left module $P_p= Ah(p)$ with grading $q \in (P_p)_{qp}$.

If $A$ is positively graded then the category of $\cC$-graded bimodules with bihomogeneous morphisms is a perfect category in the sense of Eilenberg \cite{Eilenberg} and hence we can construct a minimal projective bimodule resolution of $A$ as a bimodule over itself
with the obvious grading. The first terms of this map are
\[
\xymatrix{
\underset{s \in \cS}\bigoplus F_{s}\ar[r]^{\delta_3}&
\underset{r \in \cR}\bigoplus F_{r}\ar[r]^{\delta_2}&
\underset{b \in Q_1}\bigoplus F_{b}\ar[r]^{\delta_1}&
\underset{i \in Q_0}\bigoplus F_{i}\ar[r]^m&A
}
\]
where
\se{
m(q_1 \otimes i \otimes q_2) &= q_1q_2\\
\delta_1(q_1 \otimes b \otimes q_2) &= q_1b\otimes t(b) \otimes q_2 - q_1\otimes h(b)  \otimes bq_2\\ 
\delta_2(q_1 \otimes r \otimes q_2) &= \sum_k q_1a_1\dots \otimes a_{k} \otimes \dots a_n q_2 - 
\sum_k q_1b_1\dots \otimes b_{k} \otimes \dots b_m q_2\\
\delta_3(q_1 \otimes s \otimes q_2) &= q_1sq_2.
}
\
and $r=a_1\dots a_n - b_1\dots b_m$, $F_r := F_{a_1\dots a_n}=F_{b_1\dots b_m}$ and $\cS$ is a minimal set of homogeneous
generators of $\Ker \delta_2$. This set might be infinite. 

For every vertex $i \in Q_0$ we can tensor this resolution over $A$ on the right with the one-dimensional left module $S_i=Ai/\cJ i$ concentrated in $\cC$-degree $i$. This gives us the minimal graded left module resolution of $S_i$.

\section{Quiver polyhedra}\label{polyhedron}

\subsection{Definition}
The last ingredient we need is quiver polyhedra. 
\begin{definition}
A \emph{quiver polyhedron} $\qpol$ is a strongly connected quiver $Q$ enriched with 2 disjoint sets of cycles $Q^+_2$ and $Q_2^-$ such that
\begin{itemize}
\item[PO] {\bf Orientability condition.} Every arrow is contained exactly once in one cycle in $Q_2^+$ and once in one in $Q_2^-$. 
\item[PM] {\bf Manifold condition.} The incidence graph of the cycles and arrows meeting a given vertex is connected.
\end{itemize}
A \emph{weighted quiver polyhedron} is a pair consisting of a quiver polyhedron and
a map $\ccyc: Q_2= Q_2^+\cup Q_2^- \to \N_{>0}$
such that $\forall c \in Q_2:\ccyc_c|c|>2$.
By abuse of notation, we will use the symbol $\qpol$ to denote a weighted quiver polyhedron.
If $\ccyc$ is the constant map to $1$ we say $\qpol$ is an unweighted quiver polyhedron.

A weighted quiver polyhedron is \emph{positively graded} if there is an $\carr$-charge $\carr: Q_1 \to \R_{>0}$
such that the expression $\ccyc_c\carr_c$ is the same for all cycles in $Q_2$.  
\end{definition}
\begin{remark}
For a number of examples of quiver polyhedra we refer to section \ref{examples}.
There we also discuss the connection with dimer models.
\end{remark}
\begin{remark} 
Not every weighted quiver polyhedron can be given a grading. In \cite{broomhead}[Remark 2.3.5] a combinatorial condition
is given for this to be true in the case of unweighted quiver polyhedra. At the end of this section we will
state this condition for weighted quiver polyhedra.

The grading that one can assign to a weighted quiver polyhedron is also far from unique, however 
the algebraic properties that we will discuss further on do not depend on it. They depend merely on 
the existence of a grading.
\end{remark} 

From a quiver polyhedron we can build a topological space $X$ by associating to every cycle of length $k$ a $k$-gon. We label the edges of this $k$-gon cyclicly by the arrows of the quiver and identify edges of different polygon labeled with the same arrow. 
\begin{lemma}\label{comsurf}
If $\qpol$ is a quiver polyhedron then $X$ is a compact orientable surface.
\end{lemma}
\begin{proof}
We need to show that every point in $X$ has a neighborhood that is homeomorphic to an open disk.
For the internal points of the polygons this is trivially true. If $p$ lies on an edge of
a polygon but not on a corner then this is true because by condition PO a small enough neighborhood of $p$ will consist of
two half disks glued together. 
If $p$ is a corner of a polygon, a neighborhood of $p$ consists of triangles glued together over common edges. The result in general will be a set of disks glued together at $p$ and there is just one disk if and only if PM holds.

Using the condition PO, this surface
can be oriented by assigning an anticlockwise direction to the cycles in $Q_2^+$ and a clockwise direction for those 
in $Q_2^-$.
\end{proof}
Conversely, if $Q$ is a strongly connected quiver drawn on an orientable surface such that the complement of the quiver consists of simply connected pieces bounded by cycles, then we can give $Q$ the structure of a quiver polyhedron by
taking as $Q^+_2$ the cycles that bound pieces anticlockwise and as $Q^-_2$ the cycles that bound pieces clockwise.
It can easily be checked that PO and PM hold.

If $\qpol$ is a weighted quiver polyhedron, it is interesting to give the topological space $X$ the structure of an orbifold. 
We can do this by substituting the $k$-gon corresponding to a cycle $c$ with the orbifold
obtained by quotienting an $kr$-gon by the rotation group of order $r=\ccyc_c$. If we do this for all cycles we get an orbifold
that contains an orbifold singularity of order $\ccyc_c$ for every cycle $c$. We will denote this orbifold by $|\qpol|$.
It is clear from this construction that the orbifold $|\qpol|$ of a trivially weighted quiver is just the compact surface $X$.

For any weighted quiver polyhedron it makes sense to define its Euler characteristic as the Euler characteristictic of
its orbifold $|\qpol|$.
\[
\chi_\qpol := \# Q_0 - \# Q_1 + \sum_{c \in Q_2} \frac 1{\ccyc_c}.
\]

For a weighted quiver polyhedron $\qpol$, we can define a superpotential
\[
W = W^+-W^- :=\sum_{c \in Q_2^+}\frac {c^{\ccyc_c}}{\ccyc_c}  - \sum_{c \in Q_2^-}\frac {c^{\ccyc_c}}{\ccyc_c}+[\C Q,\C Q].
\]
Here $c^{\ccyc_c}$ stands for a cycle obtained by running through $c$ $\ccyc_c$ times.
This superpotential gives rise to a Jacobi algebra $A_{\qpol} := A_W$. Note that $W \in \cJ^3$ because for every cycle $\ccyc_c|c|>2$.
\begin{lemma}
For any (positively graded) weighted quiver polyhedron $\qpol$ the Jacobi algebra $A_{\qpol}$ is a (positively graded) category algebra.
\end{lemma}
\begin{proof}
For any arrow $a$ the partial derivative $\partial_a W$ is $\ccyc_c$ times the sum of two paths with opposite signs and therefore every relation is of the form $p-q$. 

If $\qpol$ is positively graded then the superpotential is homogeneous so the Jacobi algebra is positively graded.
\end{proof}
\begin{remark}
It is important to note that $A_{\qpol}$ is a category algebra but not always a cancellation algebra.
We will come back to this issue in section \ref{cancellation}.
\end{remark}

\subsection{Galois covers}
A morphism between weighted quiver polyhedra $\qpol^A$ and $\qpol^B$ is a pair of maps $\phi:Q_0^A \to Q_0^B$ and
$\phi:Q_1^A \to Q_1^B$ respecting head and tails ($\phi(h(a))=h(\phi(a))$ and $\phi(t(a))=t(\phi(a))$)
such that if $c \in Q_2^{A+}$ ($Q_2^{A-}$) we can find a $d \in Q_2^{B+}$ ($Q_2^{B-}$) such that $\phi(c^{\ccyc_c}) = d^{\ccyc_d}$. One can check easily from the definition that a morphism between quiver polyhedra gives corresponding morphisms between their orbifolds and their path algebras.

Let $\grp G$ be a group of automorphisms of a weighted quiver polyhedron $\qpol$ such that no nontrivial element $g \in \grp G$ fixes a vertex of $\qpol$. The quotient quiver $Q/\grp G$ is defined as the quiver with as vertices and arrows the orbit classes of vertices and arrows in $\qpol$. There is a projection map $\pi:Q \to Q/\grp G$ that maps
each vertex and arrow to its orbit. Under $\pi$ every cycle $c \in Q_2$ is mapped to a cycle in $\qpol/\grp G$.
This cycle can sometimes be the power of a smaller primitive cycle: $\pi(c)=d^k$ for some $k$. The unique way to equip 
$Q/\grp G$ with a polyhedral structure is
\[
 (Q/\grp G)_2^\pm := \{d|\exists c \in Q_2^\pm: d^k=\pi(c) \text{ and $d$ is primitive}\},
\]
The weighting has the following form
\[
\ccyc_d := k \ccyc_c.
\]

The following theorem is straightforward:
\begin{theorem}
\begin{itemize}
\item[]
 \item 
If $\grp G$ is a group of automorphisms of a weighted quiver polyhedron $\qpol$ such that no nontrivial element $g \in \grp G$ fixes a vertex of $\qpol$ then the quotient morphism $\pi:\qpol \to \qpol/\grp G$ induces a cover morphism between the
two orbifolds $\tilde\pi:|\qpol| \to |\qpol/\grp G|$ and the group of cover automorphisms of $\tilde \pi$ is $\grp G$.
\item
On the level of path algebras we have a surjective map $\pi: \C\qpol \to \C\qpol/\grp G$ such that
if $q$ is a path in $\qpol/\grp G$ then for every vertex $v \in \pi^{-1}(h(q))$ there is a unique lifted path
$\lift_v q \in \pi^{-1}(q)$ such that $h(\lift_v q)=v$.
\item
Two paths in $q_1,q_2 \in \C \qpol/\grp G$ are equivalent in $A_{\qpol/\grp G}$ if and only if
there is a $v \in \pi^{-1}h(q_1)$ such that $\lift_v q_1$ is equivalent with $\lift_v q_2$ in $A_{\qpol}$.
\end{itemize}
\end{theorem}
\begin{proof}
The first statement follows because by construction $|\qpol/G|$ is the same orbifold as the quotient orbifold $|\qpol|/\grp G$.

Suppose $q$ is a path in $\qpol/\grp G$ and choose any lift $p \in \pi^{-1}(q)$. There is a unique $g \in \grp G$ that
maps $h(p)$ to $v \in \pi^{-1}(h(q))$ and therefore the lift of $q$ starting in $v$ is $g\cdot p$.

The last statement follows from the easy to check facts that $\pi (\partial_a W_{\qpol})= \partial_{\pi{a}}W_{\qpol/\grp G}$ and $\pi^{-1} \partial_b W_{\qpol/\grp G} = \{\partial_a W_{\qpol}| \pi(a)=b\}$.
\end{proof}

This theorem states in fact that $A_{\qpol}$ is a Galois cover of $A_{\qpol/\grp G}$ in the sense of \cite{galoiscover}.
It implies a close relationship between the two Jacobi algebras and many nice properties will either hold
in both or in none. An interesting example of such a property is the cancellation property.

\begin{theorem}\label{cancelcover}
The Jacobi algebra $A_{\qpol}$ is a cancellation algebra if and only if $A_{\qpol/\grp G}$ is a cancellation algebra.
\end{theorem}
\begin{proof}
Let $p$ and $q$ be distinct paths and $a$ an arrow in $\qpol$ such that $pa=qa$.
In the quotient we have that $\pi(p)\pi(a)=\pi(q)\pi(a)$ and $\pi(p)\ne\pi(q)$ because these paths cannot be in the same orbit as they start in the same vertices and $\grp G$ acts freely on the vertices.

Suppose on the other hand that $r$ and $s$ are distinct paths and $b$ is an arrow in $\qpol/\grp G$ with $rb=sb$. Fix a vertex 
$v \in \pi^{-1}(h(r))$. By the lifting property $\lift_{v}(rb)=\lift_{v}(sb)$ and both must end in the same arrow
$a \in \pi^{-1}(b)$, so $\lift_{v}(r)a=\lift_{v}(s)a$ but  $\lift_{v}(r)\ne\lift_{v}(s)$ because their projections
to $A_{\qpol/\grp G}$ are different.
\end{proof}

The existence of a grading is compatible with the notion of Galois covers.
\begin{lemma}\label{covergrad}
$\qpol$ admits a grading if and only if $\qpol/\grp G$ does.
\end{lemma}
\begin{proof}
If $\qpol$ is positively graded, we can also give the new polyhedron a grading: 
\[
\carr_{\pi(a)} := \frac 1{|\grp G|}\sum_{b\in \grp Ga} \carr_b. 
\]
If $\qpol/\grp G$ is graded we transfer the grading as follows:
\[
\carr_{a} := \carr_{\pi(a)}. 
\]
\end{proof}

The technique of Galois covers can be used to simplify the structure of the polyhedron, without
changing the important properties of the cancellation algebra.

\begin{theorem}
A weighted quiver polyhedron can be covered by a quiver polyhedron with trivial weighting if and only if
it is not of the following forms:
\begin{itemize}
\item It has the topology of a sphere and 1 face with non-trivial weight.
\item It has the topology of a sphere and 2 faces with different non-trivial weights.
\end{itemize}
\end{theorem}
\begin{proof}
Given an orbifold $X$ with a weighted quiver polyhedron $\qpol$ on it, we can use every orbifold cover $\tilde X \to X$
to obtain a Galois cover $\tilde \qpol \to \qpol$. If $\tilde X$ is a manifold then $\tilde \qpol$ is unweighted.
From theorem 13.3.6 in \cite{thurston} we know that in dimension 2 the only orientable orbifolds that cannot be covered by a manifold
are the sphere with 1 or 2 different orbifold points. These correspond to the quiver polyhedra described above.
\end{proof}
In accordance with the theory of orbifolds we call $\qpol$ \emph{developable} if it has an unweighted galois cover.
We will denote the unweighted cover of a weighted quiver polyhedron $\qpol$ by $\qpol^u$.

This cover can be used to check whether $\qpol$ admits a grading.
\begin{lemma} 
A weighted quiver polyhedron $\qpol$ admits a grading if and only if it is developable
and its unweighted cover admits a grading.
\end{lemma} 
\begin{proof}
Suppose $\qpol$ is not developable. Then it has the topology of a sphere and has 2 cycles $u_1,u_2$  
such that all other cycles are unweighted.
For any grading compatible with the Jacobi relations we have
\[
 \carr_{u_1} = \carr_{u_2} \mod \carr_u 
\]
with $u$ an unweighted cycle.
Indeed the cycles $u_1,u_2$ have the same homology class in the union of all unweighted faces. 
Because all these faces have the same degree $\carr_u$, the difference in degree between
$u_1$ and $u_2$ must be a multiple of $\carr_u$. 
But $\carr_{u_i}=\frac{\carr_{u}}{\ccyc_{u_i}}$  so the weights of $u_1$ and $u_2$ must be the same and hence
a positive grading is impossible.

A developable quiver polyhedron admits a positive grading if and only if its unweighted cover admits a positive grading by lemma \ref{covergrad}.
\end{proof}

For unweighted quiver polyhedra we can use Hall's theorem (see \cite{broomhead}[Remark 2.5.5]) to check whether a grading exists.
\begin{theorem}[Hall]
An unweighted quiver polyhedron $\qpol$ admits a grading if and only if
for any subset $S^+ \subset (Q^u)_2^+$ we have that if $S^-\subset (Q^u)^-_2$ is the set of cycles
connected to cycles of $S^+$ then
\[
|S^+|\ge |S^-|
\]
with equality only happening if $S^+$ is not a proper subset.
\end{theorem}

If a weighted quiver polyhedron $\tilde \qpol$ is developable, then the pullback of $\tilde \qpol$ under the universal cover map is called the \emph{universal cover} of $\tilde \qpol$.
This quiver is infinite if the Euler characteristic of $|\qpol|$ is zero or negative.
It still makes sense to define the corresponding category and category algebra, however one must take
care that the latter is not a unital algebra any more. We will denote the universal cover of $\qpol$ by
$\tilde \qpol$.

\section{The CY-3 property and quiver polyhedra}\label{maintheorem}

Jacobi algebras coming from quiver polyhedra appear naturally in the context of CY-3 algebras.
\begin{theorem}\label{qpol}
If a positively graded cancellation algebra $A$ is CY-3 then it comes from a graded weighted quiver polyhedron.
\end{theorem}

To prove this theorem we need a lemma which is an adaptation of a theorem from \cite{bocklandt}.
\begin{lemma}\label{gradedcalab}
If a positively graded cancellation algebra $A=\C Q/\cI$ is CY-3 then it is a Jacobi algebra of some superpotential $W$ and there exist a coefficients $\lambda_a\in \C$ depending on $a \in Q_1$ such that
$\partial_a W = \lambda_a (p-q)$ for some $p-q \in \cR$.
\end{lemma}
\begin{proof}
We adapt the proof in \cite{bocklandt}[Theorem 3.1] which worked for an $\N$-graded algebra generated in degree $1$, to
this setting (where arrows can have different $\carr$-degree). 

As the global dimension of $A$ must be $3$, we know from section \ref{GPR} the minimal projective $\cC$-graded resolution of the trivial module
$S_i=Ai/\cJ$ with $\cC$-degree $i$ looks like
\[
\xymatrix{
P_\omega\ar@{^(->}[r]^{(f_r)}&\bigoplus_{t(r)=i}P_{r}\ar[r]^{(r\del b)}&\bigoplus_{t(b)=i}P_{b}\ar[r]^{(\cdot b)}&P_i\ar@{->>}[r]&S_i.
}
\]
In the diagram above the $r's$ are elements of the minimal set of relations $\cR$ and the $b's$ are arrows. Note that
the last term in the resolution $P_\omega$ must be isomorphic to $P_i$ because $\dim \Ext^3(S_i,S_j)\stackrel{CY}{=}\dim \Hom(S_j,S_i)=\delta_{ij}$.
This $P_i$ is shifted in $\cC$-degree, and we let $\omega$ be the path that corresponds to $i\in P_\omega$.

Consider the finite dimensional quotient algebra 
\[
M = A/(f_r: r \in \cR, A_n: n\ge N) \text{ where }\forall r :N> \carr_{f_r}.
\]
The Calabi-Yau property allows us to calculate the dimension of $iMj$:
\[
\Dim iMj = \Dim \Hom(P_i,Mj) = \Dim \Ext^3(S_i,Mj)\stackrel{CY}{=}\Dim \Hom(Mj,S_i) = \delta_{ij},  
\]
and conclude that $M$ must be isomorphic to the degree zero part of $A$. There are only as many $f_r$ as there are arrows 
($\dim \Ext^2(S_i,S_j)\stackrel{CY}{=}\dim \Ext^1(S_j,S_i)$). An $f_r$ (with $r=p-q$) cannot be a linear combination of different arrows $a$ and $b$
because this would imply that $\omega$, $ap$ and $bp$ have the same $\cC$-degree which contradicts the cancellation property. Hence, we can conclude that the $f_r$ must all be scalar multiples of arrows.

By rescaling our original relations, we can assume that the $f_r$ can be identified with the arrows. Let $r_a$ be the (nonzero) relation for which $f_{r_a}=a$. 

Because the resolution of $S_i$ is a complex we have that $\sum_{a} ar_a\del b \in \cI$ so we can write it as
\[
\sum_{h(a)=i} ar_a\del b = \sum g_{bc} r_c + rest\text{ with }rest \in \cJ\cI+\cI\cJ.
\]
If we apply property C3 we can conclude that $g_{bc}\ne 0 \iff b=c$. 
The terms $ar_a\del{b}$ all have the same $\cC$-degree which is equal to the degree of $r_b$.
As $r_b$ is a minimal relation, there are no $\cC$-homogeneous elements in $\cJ\cI+\cI\cJ$ with the same $\cC$-degree as $r_b$ and hence $rest=0$.

By introducing an appropriate rescaling of the relations we can assume that $g_{bb}=1$ and
\[
\sum_{h(a)=i} ar_a = \sum_{t(b)=i}r_bb.
\]
If we sum these equations we get a superpotential $W := \sum_a ar_a  = \sum_b r_b b$ and it is
clear that $iW$ is $\cC$-homogeneous. Note that $ar_a$ and $r_aa$ have the same $\carr$-degree
but sit in different parts of $W$ ($h(a)W$ and $Wt(a)$). We can use this together with the fact that $\qpol$ is 
strongly connected to show that $W$ is $\carr$-homogeneous.

Finally, because of the rescalings $r_a =\lambda_a (p-q)$ for some $\lambda_a \in \C$ and some $p-q \in \cR$.
\end{proof}

\begin{proof}[Proof of theorem \ref{qpol}]
Because $A$ is positively graded and CY-3 we know from lemma \ref{gradedcalab} that $A= A_W$ for
some superpotential $W$ and $\partial_{a} W=\lambda_a (p_a-q_a)$
for some scalar $\lambda_a$ and some relation $p_a-q_a \in \cR$. 

Every arrow occurs exactly in two cycles in $W$ ($ap_a$ and $aq_a$). 
If an arrow $a$ occurs in a cycle $c$ it can occur only once in this cycle or $c$ is a power of a smaller cyclic path containing just one $a$. If this were not the case, the partial derivative to $a$ of this cycle would contain more than one term with the same sign which is impossible. 

Let $Q_2$ be the set of all cycles $c$ such that a power $c^k$ occurs in $W$ and which are not powers of smaller cycles. 
The grading $\carr$ on $A$ gives a grading $\carr$ on the arrows and we define $\ccyc_c=k$ if and only if $c^k$ sits
in $W$. 

This data turns $Q$ into a weighted quiver polyhedron: 
\begin{itemize}
\item[PM] Fix a single vertex $i$ and consider the following graph $G_i$: its nodes correspond to the arrows
which have a head or a tail equal to $i$. There is an edge between two arrows $a,b$ with $t(a)=h(b)=i$ if
$ab$ is contained in a cycle of $W$. 

For every connected component $C \subset G_i$ we can construct a syzygy:
\[
z_C = \sum_{a \in C,h(a)=i} a\otimes \partial_a W \otimes 1-\sum_{a \in C,t(a)=i}1\otimes\partial_a W\otimes a.
\]
Indeed for every vertex $i$ the expression $\sigma_i := \sum_{h(a)=i} a\otimes\partial_a W\otimes 1 -\sum_{t(a)=i}1\otimes\partial_a W\otimes a$ is
a syzygy. We can split this syzygy in parts because the sets of arrows occurring in $z_{C_1}$ and $z_{C_2}$ for
two different components $C_1$ and $C_2$ are disjoint.
By the CY-3 property C2 we know that the third syzygies are in one to one correspondence with the vertices.
We can conclude that $G_v$ consist of one component. 
 \item[PO] Define the map $\coeff:Q_2\to \C$ such that $W = \sum_{c \in Q_2}\frac{\coeff(c)}{\ccyc_c}c^{\ccyc_c}$.
We will show that the image of this map is $\{\lambda,-\lambda\}$ for some $\lambda \in \C$.
We take $Q^\pm_2$ the preimage of $\pm \lambda$. Clearly if two cycles share an arrow $a$ then $\coeff(c_1)=-\coeff(c_2)=\lambda_a$.
So $\image \coeff=\{\lambda,-\lambda\}$ if we can go from one cycle to every other cycle by hopping over joint arrows.
This follows from condition PM and the fact that $Q$ is strongly connected.
\end{itemize}
The fact that $Q$ is strongly connected also implies that the $c^{\ccyc_c}$ have the same $\carr$-degree and because
$W \subset \cJ^3$ we must also have that $\ccyc_c|c|>2$. This implies that $\ccyc$ is a weighting for the quiver polyhedron and $\carr$ is a compatible grading.
\end{proof}

\section{Toric orders and dimer models}\label{cancellation}

In the previous section we proved that the CY-3 property for positively graded cancellation algebras implies the existence of
a weighted quiver polyhedron. Now we will prove that if we restrict to cancellation algebras that are toric orders, we
obtain that the quiver polyhedron must be unweighted and its underlying manifold $|\qpol|$ must be a torus. 

In order to prove this result we must first have a look at the cancellation property for quiver polyhedra.
\subsection{Cancellation for quiver polyhedra}
As we noted in section \ref{polyhedron} not all graded weighted quiver polyhedra give cancellation algebras. 

The relations in the Jacobi algebra $A_{\qpol}$ imply that all cycles in $Q_2$ are equivalent: ${c_1}^{\ccyc_{c_1}}p=p{c_2}^{\ccyc_{c_2}}$ for every $p$ with $h(p)=t(c_1)$ and $t(p)=h(c_2)$. This implies that the algebra $A$ has a central element: $\sum c^{\ccyc_c}$ where we sum over a subset representatives of $Q_2$ that contains just one cyclic path $c$  with $h(c)=i$ for every $i\in Q_0$. We will denote this central element by $\ell$. For every arrow $a$ we can find a path $p$ such that $ap=h(a)\ell$ and $pa=t(a)\ell$: just take
$p =\partial_a c^{\ccyc_c}$ where $c$ is a cycle in $Q_2$ containing $a$.

The cancellation property states that the map
\[
 A_{\qpol} \to A_{\qpol} \otimes_{\C[\ell]} \C[\ell,\ell^{-1}]
\]
is an embedding. This tensor product is the algebra obtained by making every arrow invertible (i.e. for every $a$ we have an $a^{-1}$ such that $aa^{-1}=h(a)$ and
$a^{-1} a=t(a)$). This algebra is the localization of $A_{\qpol}$ by the Ore set $\{\ell^k|k \in \N\}$ and we denote it by $\hat A_{\qpol}$. This definition makes also sense if $A_\qpol$ does not satisfy the cancellation property, the map $A_{\qpol} \to \hat A_{\qpol}$ is no more injective but $\hat A_{\qpol}$ is still a flat $A_{\qpol}$-module.

\begin{lemma}\label{localcy3}
If $A_{\qpol}$ is CY-3 then $\hat A_{\qpol}$ is also CY-3.
\end{lemma}
\begin{proof}
Let $P^\bullet$ be the bimodule resolution of $A_{\qpol}$ as a module over itself. 
The complex $\hat A_{\qpol} \otimes_{A_{\qpol}}P^\bullet \otimes_{A_{\qpol}}\hat A_{\qpol}$ is still exact because
$\hat A_{\qpol}$ is a flat $A_{\qpol}$-module. This implies that $\hat A_{\qpol}$ has a selfdual resolution
and is hence CY-3.
\end{proof}

It is important to note that $\hat A_{\qpol}$ always is a cancellation algebra
even when $A_{\qpol}$ is not. It is not always a CY-3 algebra, but in the case that $\qpol$ is graded and
$\chi_{\qpol}\le 0$ it will be even CY-3 if $A_{\qpol}$ is not. To prove this statement we first need to recall a well known lemma.

\begin{lemma}\label{homotop}
Let $\qpol$ be a weighted quiver polyhedron and $\carr:Q_1 \to \R$ be any (not necessarily positive) grading such that $\carr_\ell\ne 0$. Two paths in $\hat A_{\qpol}$ are equivalent if and only if they are homotopic and have the same $\carr$-degree.
\end{lemma}
\begin{proof}
It is clear that the relations $\partial_a W$ imply that equivalent paths are homotopic and must have the same $\carr$-charge.
Because homotopies in the quiver polyhedron are generated by substituting paths $p\to q$ such that
$pq^{-1}=\ell$, homotopic paths can only differ by a factor $\ell^k$. The degree of $\ell$ is not zero, so
if homotopic paths have the same degree they must be equal in $\hat A_{\qpol}$.
\end{proof}
\begin{remark}
By homotopic we mean homotopic as paths in $|\qpol|$ considered as an orbifold, not merely as a topological space.
\end{remark}

\begin{theorem}\label{fundgroup}
For any positively graded weighted quiver polyhedron $\qpol$, 
\[
\hat A_\qpol \cong \Mat_n(\C[\Pi]) 
\]
where $n$ is the number of vertices and $\C[\Pi]$ is the group algebra of the fundamental group of some compact 
three-dimensional manifold.
\end{theorem}
\begin{proof}
Note that because of the gradedness $\qpol$ is developable, also without loss of generality we can assume that
the $\carr$-charges are rational.

Let $|\tilde \qpol| \to |\qpol|$ be the universal cover of the orbifold $|\qpol|$ and fix a vertex $i\in Q_0$. 
To every path in $p \in i\hat A_{\tilde \qpol} i$ corresponds an element in the fundamental group of $|\qpol|$, 
which gives a cover automorphism $\phi_p: |\tilde \qpol| \to |\qpol|$. 
Conversely, every element in the fundamental group can be represented by a path in $\qpol$.

Now consider the simply connected space $|\tilde \qpol| \times \Rl$ and consider the group of diffeomorphisms
\[
\Pi = \{\psi_p: |\tilde \qpol|\times \Rl \to |\tilde \qpol|\times \Rl: (x,a) \mapsto (\phi_p(x),a+\carr_p)| p \in \Hom_{\cC_{\hat{A}_{\qpol}}}(i,i) \}
\] 
By lemma \ref{homotop}, every element in $\Hom_{\cC_{\hat A_{\qpol}}}(i,i)$ gives a different diffeomorphism and
none of these diffeomorphisms has fixpoints. The quotient of $|\tilde \qpol|\times \Rl/\Pi$ is thus a manifold and
${i\hat A_{\qpol} i} \cong \C[\Pi] =\C[\pi_1(|\tilde \qpol|\times \Rl/\Pi)]$. The manifold $|\tilde \qpol|\times \Rl/\Pi$
projects down to the compact space $|\qpol|$. The fibers of this projection are isomorphic to
$\Rl/u\Z$ where $u=\min\{|\carr_p||\Hom_{\cC_{\hat{A}_{\qpol}}}(i,i)\}$ (this minimum exists because we assumed the 
$\carr$-charges rational) so $|\tilde \qpol|\times \Rl/\Pi$ is compact.

For every vertex $j$, fix a path $p_{j}:i \ot j$. Construct the following morphism
\[
 \Mat_n(i\hat A_{\qpol} i) \to \hat A_{\qpol}: qE_{uv} \mapsto p_u^{-1}qp_v
\]
where $E_{uv}$ is the matrix with one on the entry $(u,v)$ and zero everywhere else.
This morphism has an inverse
\[
 \hat A_{\qpol} \to \Mat_n(i\hat A_{\qpol} i): q \mapsto p_{h(q)}qp_{t(q)}^{-1} E_{h(q)t(q)}.
\]
\end{proof}

We recall a theorem by Kontsevitch.
\begin{theorem}[Kontsevich, see \cite{GB} Corollary 6.1.4]
The fundamental group algebra of a compact manifold is CY-n if it is orientable and its universal cover
is contractible. 
\end{theorem}
This theorem can be used to relate the Euler characteristic of the dimer model with the CY-3 property.
\begin{corollary}\label{nonpos}
Let $\qpol$ be any positively graded weighted quiver polyhedron.
\begin{itemize}
 \item $\hat A_\qpol$ is CY-3 if and only if $\chi(\qpol)\le 0$.
 \item if $A_\qpol$ is CY-3 then $\chi(\qpol)\le 0$.
\end{itemize}
\end{corollary}
\begin{proof}
We know that $\hat A_\qpol$ is Morita equivalent to the fundamental group algebra of some $3$-manifold.
This manifold has as universal cover $|\tilde \qpol|\times \Rl$. This is contractable when $\chi_{\qpol}\le 0$.
So by Kontsevich theorem $\hat A_\qpol$ is CY-3. 
If the quiver polyhedron $\qpol$ has positive Euler characteristic, then its universal cover $\tilde\qpol$ has the topology 
of a sphere and the quotient manifold of the cover is $\SS_2 \times \SS_1$.
The fundamental group is $\Z$ so $\hat A_{\tilde\qpol}$ is Morita equivalent to $\C[\ell,\ell^{-1}]$. This last algebra is
not CY-3.

If $A_\qpol$ is CY-3 then $\hat A_\qpol$ is also CY-3 so by the previous paragraph  $\chi(\qpol)\le 0$.
\end{proof}

\begin{theorem}\label{qpolto}
If a toric order $A$ is CY-3 then it comes from a positively graded unweighted quiver polyhedron on a torus (in other words a dimer model on a torus)
\end{theorem}
\begin{proof}

Because a toric order is cancellation we already know it comes from a weighted quiver polyhedron, so we only need to show that the weights are trivial and $\chi_{\qpol}=0$.

If $A_{\qpol}$ is a toric order then $A_\qpol\subset \hat A_\qpol \subset \Mat_n(\C[\Z^3])$
because $\ell$ is invertible in $\Mat_n(\C[\Z^3])$.
Hence, for every vertex $v$, the algebra $v\hat A_{\qpol}v$ is commutative.
This means that the fundamental group of the 3-manifold is commutative and by construction
the orbifold fundamental group of the 2-orbifold must also be commutative.

As $\chi_\qpol\le 0$, this is only the case if $|\qpol|$ is a torus. Indeed, if the fundamental group 
of $|\qpol|$ is abelian then the deck transformations of $|\tilde \qpol|\to|\qpol|$ cannot have fixpoints (as such
transformations do not commute with the fixpointless ones) this means $|\qpol|$ is a manifold (not an orbifold) and $\qpol$ is unweighted. 
The only compact surface with nonpositive Euler characteristic and abelian fundamental group is the torus.
\end{proof}
 
\subsection{Cancellation and Calabi-Yau}

The cancellation property and the Calabi-Yau property are very closely related.
For quiver polyhedra with $\chi_\qpol\le 0$, Ben Davison in \cite{davison} proved that cancellation implies CY-3.

\begin{theorem}[Davison]\label{canc}
The Jacobi algebra of a graded weighted quiver polyhedron with nonpositive Euler characteristic is CY-3 if it is a cancellation algebra.
\end{theorem}
Although Davison proved this only in the case of dimer models (which are in our terminology the trivially weighted 
quiver polyhedra) his proof generalizes to the weighted case because we can cover any graded weighted quiver polyhedron 
by a graded unweighted quiver polyhedron.
Davison's work was a generalization of work by Mozgovoy and Reineke \cite{MR} which used an extra consistency condition. This extra condition turned out to be a consequence of the cancellation property.

It is not clear whether for quiver polyhedra with $\chi_\qpol\le 0$ the cancellation property is really equivalent to the CY-3 property.
There are no known examples of noncancellation CY-3 quiver polyhedra for which $\qpol$ is finite. There are however examples where $\qpol$ is infinite. 
We refer to the follow-up paper \cite{bocklandtconsistency} which discusses the different notions of consistency for quiver polyhedra. 

\section{Examples}\label{examples}

The most studied examples of quiver polyhedra come from dimer models. Dimer models are dual to unweighted quiver polyhedra.
A dimer model consists of a bipartite graph on a Rieman surface. The bipartiteness implies that vertices are coloured black and white in such a way that no vertices
of the same colour share an edge. Given a unweighted quiver polyhedron we construct
the dimer model by using the centers of the cycles in $Q_2^+$ as black vertices and   
the centers of the cycles in $Q_2^-$ as white vertices, two vertices are connected if their cycles share a face.

\begin{example}\label{pinchpoint}
The suspended pinchpoint \cite{FranHan}[section 4.1] is an example of a CY-3 algebra given by the following dimer model and quiver:
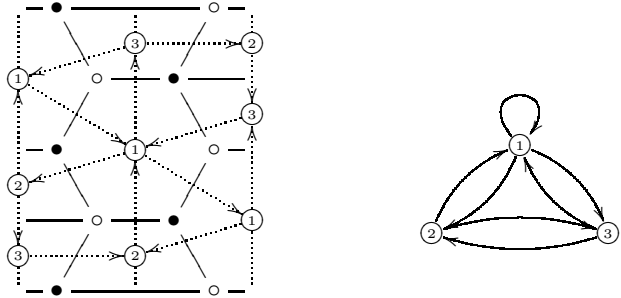
\[
\vcenter{
\xymatrix@C=.15cm@R=.125cm{
\ar@{-}[r]&\bk\ar@{-}[rdd]\ar@{-}[rrrr]&&&&\wt\ar@{-}[r]&\\
&&&\vtx{3}\ar@{.>}[rrr]\ar@{.}[u]\ar@{.>}[llld]&&&\vtx{2}\ar@{..}[u]\ar@{.>}[dd]\\
\vtx{1}\ar@{.>}[rrrdd]\ar@{.}[uu]&&\wt\ar@{-}[rr]&&\bk\ar@{-}[ruu]\ar@{-}[rdd]\ar@{-}[rr]&&\\
&&&&&&\vtx{3}\ar@{.>}[llld]\\
\ar@{-}[r]&\bk\ar@{-}[rdd]\ar@{-}[ruu]&&\vtx{1}\ar@{.>}[uuu]\ar@{.>}[llld]\ar@{.>}[rrrdd]&&\wt\ar@{-}[r]&\\
\vtx{2}\ar@{.>}[dd]\ar@{.>}[uuu]&&&&&&\\
\ar@{-}[rr]&&\wt\ar@{-}[rr]&&\bk\ar@{-}[ruu]\ar@{-}[rdd]&&\vtx{1}\ar@{.>}[llld]\ar@{.>}[uuu]\ar@{.}[dd]\\
\vtx{3}\ar@{.}[d]\ar@{.>}[rrr]&&&\vtx{2}\ar@{.>}[uuu]\ar@{.}[d]&&&\\
\ar@{-}[r]&\bk\ar@{-}[ruu]\ar@{-}[rrrr]&&&&\wt\ar@{-}[r]&
}}
\hspace{1cm}\vcenter{
\xymatrix{
&&&&\\
&&\vtx{1}\ar@(lu,ru)\ar@/^/[dr]\ar@/^/[dl]&&\\
&\vtx{2}\ar@/^/[ur]\ar@/^/[rr]&&\vtx{3}\ar@/^/[ul]\ar@/^/[ll]&
}}
\]
On the left we drew the tiling of the torus as a periodic tiling of the plane. 
The quiver is represented twice, once periodically on the left(the dotted lines) and once on the right.
There are three vertices in the quiver corresponding to the three tiles on the torus (one hexagon and two trapezia).
The sets of cycles are $Q_2^+=\{a_{31}a_{13}a_{11}, a_{21}a_{12}a_{23}a_{32}\}$ and $Q_2^-=\{a_{21}a_{12}a_{11}, a_{13}a_{31}a_{32}a_{23}\}$ and
the superpotential for this example is
\[
W = a_{31}a_{13}a_{11} + a_{21}a_{12}a_{23}a_{32} - a_{21}a_{12}a_{11}  - a_{13}a_{31}a_{32}a_{23} + [\C Q,\C Q]
\]
The arrows are indexed according to their head and tail: $h(a_{ij})=i$ and $t(a_{ij})=j$.
\end{example}

\begin{example}\label{notcancellation}
An example of a quiver polyhedron that does not give a cancellation algebra is
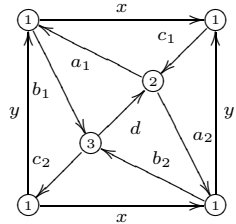
\[
 \xymatrix@C=.5cm@R=.5cm{
\vtx{1}\ar[rrr]^x\ar[ddr]_{b_1}&&&\vtx{1}\ar[ld]_{c_1}\\
&&\vtx{2}\ar[ull]^{a_1}\ar[ddr]^{a_2}&\\
&\vtx{3}\ar[ru]_d\ar[ld]_{c_2}&&\\
\vtx{1}\ar[rrr]_x\ar[uuu]^y&&&\vtx{1}\ar[ull]_{b_2}\ar[uuu]_y
}
\]
We identify all vertices labeled $\xymatrix{\vtx{1}}$ to obtain an unweighted quiver polyhedron on a torus.
The grading is done by giving all arrows degree 1.
This is not a cancellation algebra because one can check that
\[
xy \ne yx \text{ but } xy\ell=xa_1db_1y =ya_2db_2x= yx\ell.
\]
\end{example}

\begin{example}\label{octahedron}
The same quiver polyhedron can be weighted differently to obtain different Jacobi algebras.
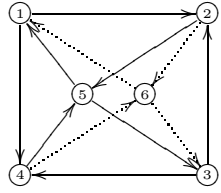
\[
\xymatrix@C=.5cm@R=.75cm{
\vtx{1}\ar[dd]\ar[rrr]&&&\vtx{2}\ar[dll]\ar@{.>}[dl]\\
&\vtx{5}\ar[ul]\ar[drr]&\vtx{6}\ar@{.>}[ull]\ar@{.>}[dr]&\\
\vtx{4}\ar[ur]\ar@{.>}[urr]&&&\vtx{3}\ar[uu]\ar[lll]
}
\]
Use the same indexing convention as in the first example and set 
\se{
Q_2^+&=\{a_{15}a_{52}a_{21},a_{43}a_{35}a_{54},a_{41}a_{16}a_{64},a_{23}a_{36}a_{62}\}\\
Q_2^-&=\{a_{16}a_{62}a_{21},a_{43}a_{36}a_{64},a_{41}a_{15}a_{54},a_{23}a_{35}a_{52}\}.
}
This polyhedron is an octahedron and has the topology of a sphere. 

\begin{itemize}
 \item 
It can be given a trivial weighting if we give the arrows degree $1$.
In this case you get a cancellation algebra (there is at most one path between every pair of vertices of a given degree) 
but it is not CY-3 because $\chi_\qpol$ is $2$.
\item
It can be equipped with a nontrivial weighting by giving  
cycles containing $a_{41}$ or $a_{23}$ weight $2$ and the rest weight $1$.
The grading gives $a_{21}$ or $a_{43}$ degree $4$ and the rest degree $1$. 
This weighting gives rise to an orbifold with Euler characteristic $0$.
One can check using results from \cite{bocklandtconsistency} (i.e. intersecting zigzag paths) that this does not give you a cancellation algebra.
\item
We can equip this with a nontrivial weighting by giving  
all cycles weight $2$ and all arrows degree $1$.
This weighting gives rise to an orbifold with Euler characteristic $-2$.
The nonintersection of zigzag paths tells us that this is a cancellation algebra and a CY-3 algebra.
\end{itemize}

\end{example}

\begin{example}\label{ASregular}
Weighted quiver polyhedra can also be used to describe certain Artin-Schelter regular algebras.
Take the following quiver polyhedron on the sphere
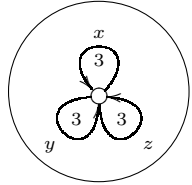
\begin{center}
\begin{tikzpicture}
\draw (1.2,0) arc (0:360:1.2);
\draw (0,0) node {$\xymatrix{\vtx{}\ar@(ru,lu)_x^3\ar@(l,d)_y^3\ar@(d,r)_z^3}$};
\end{tikzpicture}
\end{center}
where the backside is a triangle bounded by $x,y$ and $z$.
Then $A_\qpol=\C\<x,y,z\>/\<x^2-yz,y^2-zx,z^2 -xy\>$ which is a well-known three-dimensional Artin Schelter regular ring \cite{ASreg}.
The center of this ring is isomorphic to $\C[u^3,v^3,w^3,uvw]$, which is the quotient singularity of the $\grp G = \Z_3\times \Z_3$-action
$gu = \eta^i u$, $gv = \eta^{i+j} v$, $gw = \eta^{i+2j}w$ where $g=(i,j)\in \grp G$ and $\eta$ is a primitive cube root. 
It is a cancellation algebra and CY-3, but take care, this algebra $A_\qpol$ cannot be seen as a noncommutative crepant resolution over its center because a commutative crepant resolution
has a rank 9 K-group, while the K-group of $A_\qpol$ is only rank 1.
\end{example}

\bibliographystyle{amsplain}
\def\cprime{$'$}
\providecommand{\bysame}{\leavevmode\hbox to3em{\hrulefill}\thinspace}
\providecommand{\MR}{\relax\ifhmode\unskip\space\fi MR }
\providecommand{\MRhref}[2]{%
  \href{http://www.ams.org/mathscinet-getitem?mr=#1}{#2}
}
\providecommand{\href}[2]{#2}

\end{document}